\documentclass{amsart}
\usepackage{amsmath}
\usepackage{enumerate}
\usepackage{amssymb}
\usepackage[all,cmtip]{xy}

\numberwithin{equation}{section}
\theoremstyle{definition}
\newtheorem{definition}{Definition}[section]
\theoremstyle{plain}
\newtheorem{proposition}[definition]{Proposition}

\newtheorem{lemma}[definition]{Lemma}
\newtheorem{theorem}[definition]{Theorem}
\newtheorem{corollary}[definition]{Corollary}
\theoremstyle{remark}

\begin{document}

\newcommand{\sheaf}[1]{\ensuremath{\mathcal{#1}}}
\newcommand{\maxid}{{\mathop{\mathfrak{m}}}}
\newcommand{\depth}[1]{\mathop{\mathrm{depth}}\nolimits _{#1}}
\newcommand{\image}{{\mathop{\mathrm{Im\,}}}}
\newcommand{\simage}{{\mathop{\mathrm{Im\,}}}}
\newcommand{\ann}{{\mathop{\mathrm{ann\,}}}}
\newcommand{\sker}{{\mathop{\mathrm{ker\,}}}}
\newcommand{\coker}{{\mathop{\mathrm{coker\,}}}}
\newcommand{\scoker}{{\mathop{\mathrm{coker\,}}}}
\newcommand{\id}{{\mathop{\mathrm{id}}}}
\newcommand{\supp}{\mathop{\mathrm{supp}}}

\newcommand{\Hom}{{\mathop{\mathrm{Hom}}}}
\newcommand{\shom}{{\mathop{\mathbf{Hom}}}}
\newcommand{\shomo}{{\mathop{\mathbf{Hom}_\sheaf{O}}}}
\newcommand{\shomp}{{\mathop{\mathbf{Hom}_\mathrm{plain}}}}
\newcommand{\shommin}{{\mathop{\mathbf{Hom}_\mathrm{min}}}}

\title{Homomorphisms of infinitely generated analytic sheaves}

\author{Vakhid Masagutov}
\address{
	Department of Mathematics, Purdue University \\
	West Lafayette, IN 47907, USA
	}
\email{
	vmasagut@math.purdue.edu
	}

\thanks{
Research partially supported by NSF grant DMS0700281
and the Mittag-Leffler Institute, Stockholm. I am grateful to both
organizations.  I would like to express my gratitude to the 
Mittag-Leffler Institute for their hospitality during the preparation of this paper. 
I am also very grateful to Professor Lempert
for proposing questions that led to this work and for his
invaluable suggestions and critical remarks. 
}

\begin{abstract}
We prove that every homomorphism 
$\mathcal{O}^E_\zeta\rightarrow\mathcal{O}^F_\zeta$, 
with $E$ and $F$ Banach spaces and $\zeta\in\mathbb{C}^m$, is induced 
by a $\mathop{\mathrm{Hom}}(E,F)$-valued holomorphic germ, provided that $1\leq m<\infty$. 
A similar 
structure theorem is obtained for the homomorphisms
of type $\mathcal{O}^E_\zeta\rightarrow\mathcal{S}_\zeta$, where $\mathcal{S}_\zeta$ is
a stalk of a coherent sheaf of positive 
$\mathfrak{m}_\zeta$-depth.
We later extend these results to sheaf homomorphisms,
 obtaining a  condition on coherent sheaves
 which guarantees the sheaf to be equipped with a unique analytic structure in the sense of Lempert-Patyi.
\end{abstract}


\maketitle

\section{Introduction}
The theory of coherent sheaves is one of the deeper and most developed subjects
 in complex analysis and geometry, see \cite{GR}.
Coherent sheaves are locally finitely generated. However,
a number of problems even in finite dimensional geometry leads to sheaves
that are not finitely generated over the structure sheaf $\sheaf{O}$,
such as the sheaf of holomorphic germs valued in a Banach space;
and in infinite dimensional problems infinitely generated sheaves are the rule rather
than the exception. This paper is motivated by \cite{LP},
that introduced and studied the class of
so called cohesive sheaves over Banach spaces; but here we shall almost exclusively
deal with sheaves over $\mathbb{C}^m$. In a nutshell, we show that 
\sheaf{O}-homomorphisms among certain sheaves of \sheaf{O}-modules
have strong continuity properties, and in fact arise by a simple construction.

We will consider two types of sheaves. The first type consists of 
coherent sheaves. The other consists of plain sheaves; these are
the sheaves $\sheaf{O}^E$ of holomorphic germs valued in some fixed
complex Banach space $E$. The base of the sheaves is 
$\mathbb{C}^m$ or an open $\Omega\subset\mathbb{C}^m$.
Thus $\sheaf{O}^E$ is a (sheaf of) \sheaf{O}-module(s).
We denote the Banach space of continuous linear operators between
Banach space $E$ and $F$ by $\Hom(E,F)$.
Any holomorphic map $\Phi:\Omega\rightarrow\Hom(E,F)$ induces
an \sheaf{O}-homomorphism $\phi:\sheaf{O}^E\rightarrow\sheaf{O}^F$.
If $U\subset\Omega$ is open, $\zeta\in U$ and a holomorphic 
$e:U\rightarrow E$ represents a germ $e_\zeta\in\sheaf{O}^E_\zeta$,
then $\phi(e)\in\sheaf{O}^{F}_\zeta$ is defined as the germ of the function
$U\ni z\mapsto \Phi(z)e(z)\in E$.
Following \cite{LP}, such homomorphisms will be called plain.
In fact, if $\Phi$ is holomorphic only on
some neighborhood of $\zeta$ it still defines a homomorphism 
$\sheaf{O}^E_\zeta\rightarrow\sheaf{O}^F_\zeta$ of the local modules over
the local ring $\sheaf{O}_\zeta$. Again such homomorphisms will be called
plain.

The first question we address is a how restrictive it is for a homomorphism
to be plain. It turns out it is not restrictive at all, provided
$0<m<\infty$.
\begin{theorem}\label{thm_plain}
If $0<m<\infty$ and $\Omega\subset\mathbb{C}^m$ is open, then every $\sheaf{O}$-homomorphism $\sheaf{O}^E\rightarrow\sheaf{O}^F$ of plain sheaves is plain.
\end{theorem}

This came as a surprise, because it fails in the simplest of all cases, when
$m=0$. This was pointed out by Lempert.
When $\Omega=\mathbb{C}^0=\{0\}$, $\sheaf{O}^E$, resp. $\sheaf{O}^F$,
are identified with $E$ and $F$, and the difference between 
$\sheaf{O}$-homomorphism and plain homomorphism boils down to the 
difference between linear and continuous linear operators $E\rightarrow F$.
It would be interesting to decide whether Theorem \ref{thm_plain} 
remains true if $\mathbb{C}^m$ is replaced by a Banach space.

Lempert observed that a variant of the original proof
of Theorem \ref{thm_plain} gives the corresponding theorem about
local modules, and we shall derive Theorem \ref{thm_plain} from it:
\begin{theorem}\label{thm_local_modules}
If $0<m<\infty$ and $\zeta\in\mathbb{C}^m$, then every
$\sheaf{O}_\zeta$-homomorphism of plain modules 
$\sheaf{O}^E_\zeta\rightarrow\sheaf{O}^F_\zeta$ is plain.
\end{theorem}

Next we turn to the structure of \sheaf{O}-homomorphisms from plain
sheaves $\sheaf{O}^E$ to a coherent sheaf $\sheaf{S}$.
On the level of stalks, such homomorphisms have a simple description;
however, this description applies only if the depth of the stalk 
$\sheaf{S}_\zeta$ is positive, a condition that corresponds to the
positivity of $m$ in Theorems \ref{thm_plain} and \ref{thm_local_modules}.
For our purposes we can define depth as follows.
Let $\maxid_\zeta\subset\sheaf{O}_\zeta$ denote the maximal
ideal consisting of germs that vanish at $\zeta$; assume $\maxid_\zeta\not=0$. For a finitely generated
$\sheaf{O}_\zeta$-module $M$, $\depth{\maxid_\zeta} M=0$ if
$M$ has a nonzero submodule $N$ such that $\maxid_\zeta N=0$
(see Definition \ref{defn_depth} and Proposition 
\ref{prop_depth_lemmas_depth_better}).
Otherwise $\depth{\maxid_\zeta} M>0$.
\begin{theorem}\label{thm_coh_stalk}
Let $\zeta\in\mathbb{C}^m$, $M$ a finite $\sheaf{O}_\zeta$-module
with $\depth{\maxid_\zeta} M>0$, and $p:\sheaf{O}^n_\zeta\rightarrow M$
an epimorphism. Then, any $\sheaf{O}_\zeta$-homomorphism
$\phi:\sheaf{O}_\zeta^E\rightarrow M$ factors through $p$, i.e., $\phi=p\psi$
with an $\sheaf{O}_\zeta$-homomorphism $\psi:\sheaf{O}_\zeta^E\rightarrow \sheaf{O}^n_\zeta$.
\end{theorem}
Since our depth condition eliminates the possibility of a nonzero sheaf when $m=0$, the above  $\psi$
is then induced by a germ in $\sheaf{O}_\zeta^{\Hom(E,F)}$.
The depth condition is in fact necessary as shown in Theorem 
\ref{thm_necessary}.

A global version of the theorem concerning epimorphisms 
$p:\sheaf{O}^n\rightarrow\sheaf{S}$ on a coherent sheaf, also holds,
but we shall discuss it only in Section 8,
since it depends on result of Lempert that has not yet been published.
Theorem \ref{thm_coh_stalk} can be recast in the language of
analytic structure on sheaves, as defined in \cite{LP};
it says that analytic structures of coherent sheaves are unique.
This will be explained in Section 7, along with the following corollary
of Theorem \ref{thm_local_modules}:
\begin{corollary}\label{cor_non_free}
If $\zeta\in\mathbb{C}^m$, $0<m<\infty$, and $E$ is an infinite dimensional Banach space, then the plain module $\sheaf{O}^E_\zeta$ is not free;
it cannot even be embedded in a free module.
\end{corollary}

\section{Background}
Here we quickly review a few notions of complex analysis. For more see
\cite{GR,mujica,serre}.
Let $X,E$ be Banach spaces (always over $\mathbb{C}$) and $\Omega\subset X$ open.
\begin{definition}
A function $f:\Omega\rightarrow E$ is holomorphic if for all $x\in\Omega$
and $\xi\in X$
\[
df(x,\xi)=\lim_{\lambda\rightarrow 0}\frac{f(x+\xi\lambda)-f(x)}{\lambda}
\]
exists, and depends continuously on $(x,\xi)\in\Omega\times X$.
\end{definition}

If $X=\mathbb{C}^m$ with coordinates $(z_1,\ldots,z_m)$, then this is equivalent
to requiring that in some neighborhood of each $a\in\Omega$ one
can expand $f$ in a uniformly convergent power series 
\[
	f=\sum_J e_J (z-a)^J, \qquad e_J\in E,
\]
where multi-index notation is used. For general $X$ one can only talk about
homogeneous expansion. Recall that a function $P$ between vector spaces $V, W$
is an $n$-homogeneous polynomial if $P(v)=l(v,v,\ldots,v)$ where $l:V^n\rightarrow W$ is an $n$-linear map.
Given a ball $B\subset X$ centered at $a\in X$, any holomorphic $f:B\rightarrow E$ can be expanded in a series
\begin{equation}\label{eq_background_homo_exp}
	f(x)=\sum_{n=0}^\infty P_n(x-a), \qquad x\in B,
\end{equation}
where the $P_n:X\rightarrow E$ are continuous $n$-homogeneous polynomials.
The homogeneous components $P_n$ are uniquely determined, and the series 
\eqref{eq_background_homo_exp} converges locally uniformly on $B$.

We denote by $f_x$ the germ at $x\in\Omega$ of a function $\Omega\rightarrow E$,
and by $\sheaf{O}^E$ the sheaf of germs of holomorphic functions $U\rightarrow E$, where $U\subset\Omega$ is open. The sheaf $\sheaf{O}^\mathbb{C}=\sheaf{O}$
is a sheaf of rings over $\Omega$, and $\sheaf{O}^E$ is, in an obvious way,
a sheaf of $\sheaf{O}$-modules.
The sheaves $\sheaf{O}^E$ are called plain sheaves, and their stalks $\sheaf{O}^E_x$ plain modules. When $E=\mathbb{C}^n$, we write $\sheaf{O}^n$,
resp. $\sheaf{O}^n_x$ for $\sheaf{O}^E$, resp. $\sheaf{O}^E_x$.

As said in the Introduction, $\Hom(E,F)$ denotes the space of continuous linear operators between Banach spaces $E$ and $F$, endowed with the operator norm. 
Any holomorphic function $\Phi:\Omega\rightarrow\Hom(E,F)$ induces an $\sheaf{O}$-homomorphism $\sheaf{O}^E\rightarrow\sheaf{O}^F$ and any $\Psi\in\sheaf{O}_x^{\Hom(E,F)}$
induces an $\sheaf{O}_x$-homomorphism $\sheaf{O}^E_x\rightarrow\sheaf{O}^F_x$.
The homomorphisms obtained in this manner are called plain homomorphisms.

\section{Homomorphisms of Plain Sheaves and Modules}
We shall deduce Theorem \ref{thm_local_modules} from a weaker variant,
which, however, is valid in an arbitrary Banach space:
\begin{theorem}\label{thm_local_partial}
Let $X,E,F$ be Banach spaces, $\dim X>0$. Let $\zeta\in X$ and
$\phi:\sheaf{O}^E_\zeta\rightarrow\sheaf{O}^F_\zeta$ an 
$\sheaf{O}_\zeta$-homomorphism.
Then there is a plain homomorphism 
$\psi:\sheaf{O}^E_\zeta\rightarrow\sheaf{O}^F_\zeta$
that agrees with $\phi$ on constant germs.
\end{theorem}
We need two auxiliary results to prove this.
\begin{proposition}\label{prop_aux_local_partial_1}
Let $X,G$ be Banach spaces and $\pi_n:X\rightarrow G$ continuous homogeneous
polynomials of degree $n=0,1,2,\ldots$. If for every $x\in X$ there is
an $\epsilon_x>0$ such that $\sup_n \|\pi_n(\epsilon_xx)\|<\infty$,
then there is an $\epsilon >0$ such that 
\[
	\sup_n\sup_{\|x\|<\epsilon} \|\pi_n(x)\|<\infty.
\]
\end{proposition}
Here, and in the following, we indiscriminately use $\|\cdot\|$ for the
norms on $X,G$, and whatever Banach spaces we encounter.
\begin{proof}
For numbers $A$ and $\delta$, consider the closed sets
\[
X_{A,\delta}=\{x\in X:\sup_n\|\pi_n(\delta x)\|\leq A\}.
\]
By Baire's theorem $X_{A,\delta}$ contains a ball
$\{x_0+y:\|y\|<r\}$ for some $A,\delta,r>0$. As a consequence of the
polarization formula \cite[1:10]{mujica},
\[
\pi_n(\xi)=\sum_{\sigma_j=\pm 1} \frac{\sigma_1\ldots\sigma_n}{2^nn!}
	\pi_n(\delta x_0+\sigma_1\xi+\ldots+\sigma_n\xi), \qquad\text{for }\xi\in X,
\]
see also \cite[Exercise 2M]{mujica}. Therefore if $\|\xi\|<\delta r/n$, 
then
$\pi_n(\xi)\leq A/n!$, and by homogeneity, for $\|x\|<\delta r/e$
\[
\|\pi_n(x)\|=n^ne^{-n}\|\pi_n(ex/n)\|\leq An^n/(e^nn!)\leq A.
\]
\end{proof}

\begin{proposition}\label{prop_weakstar_holo_is_holo}
Let $X,E,F$ be Banach spaces, $\Omega\subset X$ open, and 
$g:\Omega\rightarrow\Hom(E,F)$ a function.
If for every $v\in E$ the function $gv:X\rightarrow F$ is holomorphic,
then $g$ itself is holomorphic.
\end{proposition}
\begin{proof}
This is Exercise 8.E in \cite{mujica}. First one shows using the Principle 
of Uniform Boundedness that $g$ is locally bounded. Standard one
variable
Cauchy representation formulas then show $g$ is continuous and 
ultimately holomorphic.
\end{proof}

\begin{proof}[Proof of Theorem \ref{thm_local_partial}]
If $v\in E$ we write $\tilde v\in\sheaf{O}^{E}_\zeta$ for the constant germ
whose value is $v$. Without loss of generality we can take $\zeta=0$.
Let the germ $\phi(\tilde v)\in\sheaf{O}^F_0$ have homogeneous series
\begin{equation} \label{eq_proof_homo_exp}
	\sum_{n=0}^\infty P_n(x,v).
\end{equation}
Thus, $P_n$ is $\mathbb{C}$-linear in $v$, and for fixed $v$, $P_n(\cdot,v)$
is a continuous $n$-homogeneous polynomial. For each 
$v\in E$ \eqref{eq_proof_homo_exp} converges if $\|x\|$ is sufficiently small.

Now let $\lambda\in \Hom(X,\mathbb{C})$, and suppose that with 
$v_j\in E$ the series $\sum_{i=0}^\infty v_i\lambda^i$ represents
a germ $e\in\sheaf{O}^{E}_0$.
For example, this will be the case if the $v_i$ are unit vectors. 
With an arbitrary $N\in\mathbb{N}$ and some $f\in\sheaf{O}^F_0$
\begin{align*}
	\phi(e)	&=\sum_{i<N} \phi(\tilde v)\lambda^i+\lambda^N f\\
		&=\sum_{j<N} \sum_{n=0}^\infty 
			P_n(\cdot,v_i)\lambda^i+\lambda^Nf\\
		&=\sum_{j<N}\sum_{n=0}^j P_n(\cdot,v_{j-n})\lambda^{j-n}
			+\lambda^Ng, \qquad\text{where }g\in\sheaf{O}^{F}_0.
\end{align*}
Hence the homogeneous components of $\phi(e)$ are 
\begin{equation}\label{eq_homo_comp}
Q_j(x)=\sum_{n=0}^j P_n(x,v_{j-n})\lambda^{j-n}(x), \qquad j=0,1,2,\ldots
\end{equation}
We use this to prove, by induction on $n$, that for any $x\in X$ the map
$v\mapsto P_n(x,v)$ is not only linear but also continuous.

Suppose this is true for $n<k$. Take an $x\in X$, which can be supposed to be
nonzero, and $\lambda\in\Hom(X,\mathbb{C})$ so that $\lambda(x)=1$.
If $v\mapsto P_k(x,v)$ were not continuous, we could inductively select
unit vectors $v_i\in E$ so that
\[
	\begin{split}
	\|P_k(x,v_{j-k})\| > 
		\sum_{n=0}^{k-1}
		\|P_n(x,\cdot)\| & +j^j
		+\sum_{n=k+1}^j \|P_n(x,v_{j-n})\| ,
	\end{split}
\]
for $j=k,k+1,\ldots$. Here $\|P_n(x,\cdot)\|$ stands for the operator
norm of the homomorphism $P_n(x,\cdot)\in \Hom(E,F)$, $n<k$.
However, \eqref{eq_homo_comp} would then imply
\[
\begin{split}
\|Q_j(x)\|\geq 
	\|P_k(x,v_{j-k}) )\| 
	-\sum_{n=0}^{k-1}\|P_n(x,\cdot)\| 
	-\sum_{n=k+1}^j \|P_n(x,v_{j-n})\|  
	>j^j,
\end{split}
\]
which would preclude $\sum Q_j$ from converging in any neighborhood of $0\in X$.
The contradiction shows that $P_k(x,\cdot)\in\Hom(E,F)$, in fact for every
$k$ and $x\in X$.
Let us write $\pi_k(x)$ for $P_k(x,\cdot)$.

Now, for fixed $v\in E$, $P_n(\cdot,v)=\pi_nv$ is holomorphic.
We can apply Proposition \ref{prop_weakstar_holo_is_holo} to
conclude that $\pi_n:X\rightarrow \Hom(E,F)$ is a holomorphic,
$n$-homogeneous polynomial.

Next we estimate $\|\pi_n(x)\|$ for fixed $x\in X$. 
Suppose a sequence $\delta_n\geq 0$ goes to 0 super-exponentially,
in the sense that $\delta_n=o(\epsilon^n)$ for all $\epsilon>0$.
Then for any homogeneous series $\sum p_n$ representing a germ 
$f\in \sheaf{O}^F_0$ we have $\sup_n \delta_n\|p_n(x)\|<\infty$.

In particular, $\sup_n \delta_n\|\pi_n(x)v\|<\infty$ for all $v\in E$,
and by the Principle of Uniform Boundedness, $\delta_n\|\pi_n(x)\|$
is bounded. This being so, there is an $\epsilon=\epsilon_x>0$
such that $\epsilon^n\|\pi_n(x)\|$ is bounded. Indeed,
otherwise we could find $n_1<n_2<\ldots$ so that 
\[
	\|\pi_{n_t}(x)\|>t^{n_t}, \text{ } t=1,2,\ldots.
\]
But then the sequence 
\[
	\delta_n=\left\{\begin{array}{cc} 	t^{-n_t/2} &\text{if }n=n_t,\\
				     	0,&\text{otherwise,}
			\end{array}\right.
\]
would go to $0$ super-exponentially and yet $\delta_{n_t}\|\pi_{n_t}(x)\|\rightarrow\infty$; a contradiction.

Thus, for each $x$ we have found $\epsilon_x>0$ so that $\sup_n\|\pi_n(\epsilon_x x)\|$  is bounded. By Proposition \ref{prop_aux_local_partial_1},
the $\pi_n$ are uniformly bounded on some ball $\{x:\|x\|<\epsilon\}$.
Therefore the series 
\[
	\sum_{n=0}^\infty \pi_n(x)=\Phi(x)
\]
converges uniformly on some neighborhood of $0\in X$, and represents a 
$\Hom(E,F)$-valued holomorphic function there.
By the construction of $\pi_n(x)v=P_n(x,v)$, see \eqref{eq_proof_homo_exp},
the plain homomorphism 
$\sheaf{O}^E_0\rightarrow \sheaf{O}^F_0$
induced by $\Phi$ agrees with $\phi$ on constant germs $\tilde v$, and the proof
is complete.
\end{proof}

\begin{proof}[Proof of Theorem \ref{thm_local_modules}]
In view of Theorem \ref{thm_local_partial}, all we have to show is that 
(for $X=\mathbb{C}^m, 0<m<\infty$) if an $\sheaf{O}_\zeta$-homomorphism
$\phi:\sheaf{O}^E_\zeta\rightarrow\sheaf{O}^F_\zeta$ annihilates
constant germs then it is in fact 0.
This we formulate in a slightly greater generality:
\begin{lemma}\label{lemma_im_is_small}
Let $\zeta\in\mathbb{C}^m$ and $M$ be an $\sheaf{O}_\zeta$-module.
If a homomorphism $\theta:\sheaf{O}^F_\zeta\rightarrow M$ annihilates
all constant germs, then 
\begin{equation}\label{eq_im_is_small}
\image \theta\subset\bigcap_{k=0}^\infty \maxid_\zeta^k M.
\end{equation}
\end{lemma}
\begin{proof}
Along with constants, $\theta$ will annihilate
the $\sheaf{O}_\zeta$-module generated by constants, in particular, 
the polynomial germs.
Since any $e\in\sheaf{O}^E_\zeta$ is congruent, modulo an arbitrary power
of $\maxid_\zeta^k$ of the maximal ideal to a polynomial, and furthermore,
$\theta(\maxid_\zeta^k\sheaf{O}^E_\zeta)\subset\maxid_\zeta^kM$,
\eqref{eq_im_is_small} follows.
\end{proof}
This then completes the proof of Theorem \ref{thm_local_modules}
since $\bigcap_{k=0}^\infty \maxid_\zeta^k\sheaf{F}_\zeta=0$.
\end{proof}

\begin{proof}[Proof of Theorem \ref{thm_plain}]
Let $\phi:\sheaf{O}^E\rightarrow\sheaf{O}^F$ be an $\sheaf{O}$-homomorphism.
For $v\in E$ let $\hat v:\Omega\rightarrow\sheaf{O}^E$ be the section that
associates with $z\in\Omega$ the germ at $z$ of the constant function $e\equiv v$.
Then $\phi(\hat v)$ is a section of $\sheaf{O}^F$ and so there is a 
holomorphic function $f(\cdot,v):\Omega\rightarrow F$ whose
germs $f(\cdot,v)_z$ at various $z\in\Omega$ agree with 
$\phi(\tilde v)(z)$. By Theorem \ref{thm_local_modules} for each
$\zeta\in\Omega$  we can find a germ $\Phi^\zeta\in\sheaf{O}^{\Hom(E,F)}_\zeta$
 such that 
\[
f(\cdot,v)_\zeta=\Phi^\zeta v
\]
Therefore for fixed $\zeta$, $f(\zeta,\cdot)\in\Hom(E,F)$.
Let $\Phi(\zeta)=f(\zeta,\cdot)$. Proposition \ref{prop_weakstar_holo_is_holo}
implies that $\Phi:\Omega\rightarrow\Hom(E,F)$ is holomorphic and by construction
induces $\phi$ on constant germs.

This means that $\Phi^\zeta$ above and the germ $\Phi_\zeta$ of $\Phi$
induce homomorphisms $\sheaf{O}^E_\zeta\rightarrow\sheaf{O}^F_\zeta$
that agree on constant germs; since we are talking about plain homomorphisms, the
two induced homomorphisms in fact agree. Hence $\phi$ is induced by $\Phi$.
\end{proof}

\section{Depth  Lemmas}
The usual definition of depth, see \cite[pp. 423,429]{eisen} or 
\cite[p.  130]{matsumura}, 
gives the following 
\begin{definition}\label{defn_depth}
Let $R$ be a Noetherian ring (always commutative, unital), $I\subset R$
and ideal and $M$ a finite $R$-module, such that $M\not=IM$. 
We say the $I$-depth 
of $M$, $\depth{I} M$, is positive if there is 
a nonzerodivisor $r$ on $M$ with $r\in I$ and $rM\not=M$;
otherwise the $I$-depth is $0$.

If $M=IM$, in particular, if $M=0$, the convention is that the $I$-depth is positive infinity.
\end{definition}
We note that, when $R$ is a field, the maximal ideal is $\maxid=0$  and so,
if $M$ is a finite $R$-module, then
the $\maxid$-depth 
 of $M$ is positive (infinity) if and
only if $M=0$.
When $R$ is  not a field, there is 
an alternative criterion for the positivity of the depth. 
While this lemma is not new, 
we include it for the sake completeness:
\begin{proposition}\label{prop_depth_lemmas_depth_better}
For a Noetherian local ring $(R,\maxid)$, not a field, a finite $R$-module $M$ 
has 
$\depth{\maxid} M=0$ if and only  if there is a nonzero submodule 
$L\subset M$ such that $\maxid L=0$.
\end{proposition}
\begin{proof}
Assume $\depth\maxid M=0$. Then,  $M\not=\maxid M$ and every
 $r\in\maxid\backslash\{0\}$ is a zerodivisor on $M$.
At this point we recall the notion of an associated prime: 
if $R$ is a commutative
ring and $M$ is an $R$-module, then a prime ideal $p$ of $R$ is associated
to $M$, if there is an $x\in M$ such that $p=\ann x$.
We shall make use of the following fact:
if $R$ is a Noetherian ring and $M$ is a finite $R$-module, then 
 there are  finitely many primes
associated to $M$, and furthermore, each zerodivisor on $M$ is contained in
one of them, see
\cite[Theorem 3.1]{eisen}.

Thus, in our setting, 
$\maxid$ is a subset of the finite union of the associated primes.
Now, the prime avoidance lemma \cite[Theorem 3.3]{eisen} states
that $\maxid$ is contained in one of the associated primes,
and hence, itself 
is an associated prime. Therefore, 
 $\maxid x=0$ for some nonzero $x\in M$ and we  put $L=xR$. 

Conversely, suppose $L\subset M$ is a nonzero submodule such that 
$\maxid L=0$. Then, every $r\in \maxid\backslash\{0\}$ is a zerodivisor.
Since $M\not=0$ is a finite $R$-module, Nakayama's lemma 
\cite[Corollary 4.8]{eisen} implies that $M\not=\maxid M$.  So,
$\depth{\maxid} M=0$.
\end{proof}

For the proof of Theorem \ref{thm_coh_stalk} we shall need 
a number of Lemmas that are algebraic in nature. Recall
the notion of localization at a prime. Suppose $R$ is a ring,
$p\subset R$ a prime ideal, and $M$ an $R$-module.
Consider the multiplicatively closed set $S=R\backslash p$,
then the localization of $M$ at $p$ is
\[
	M_p=M\times S/\sim\, ,
\]
where $(v,s)\sim(w,t)$ means $q(vt-ws)=0$ for some $q\in S$. 
Elements of $M_p$ are written as fractions $v/s$. The usual rules for
operating with fractions turn $R_p$ into a ring and $M_p$ into a module over it.
Localization is a functor, in particular, a homomorphism 
$\alpha:M\rightarrow M'$ of $R$-modules induces a homomorphism 
$\alpha_p:M_p\rightarrow M'_p$.

\begin{lemma}\label{lemma_depth_lemmas_free_loc}
Let $R$ be a unique factorization domain and 
$(p)\subset R$  a principal prime ideal.
If $N\subset R^m$ is a finite module,
then there is a free submodule $F\subset N$ such that $F_{(p)}=N_{(p)}$.
\end{lemma}
\begin{proof}
Any element of $R_{(p)}$ is either invertible or divisible by $p$. 
Since $R$ is a UFD, it follows that
given $u_1,\ldots,u_k\in R_{(p)}^m$
any non-trivial linear relation 
\[
	\sum r_ju_j=0,\text{ with }r_j\in R_{(p)},
\]
can be solved for some $u_j$. 
Hence finitely generated submodules of $R_{(p)}^m$ are free. In particular,
$N_{(p)}$ has a free generating set $(v_j/s_j)$, $j=1,\ldots,k$. We can therefore,
take $F$ to be the module generated by $v_j$'s.
\end{proof}
\begin{lemma}\label{lemma_depth_lemmas_reduce_to_torsion}
Let $(R,\maxid)$ be a local ring which is a unique factorization domain, and
$Q$ its field of fractions. Let $\rho:A\rightarrow B$ be a homomorphism of
finite free $R$-modules. If $\depth{\maxid} \coker \rho>0$, then there are a finite free $R$-module $C$ and a homomorphism $\theta:B\rightarrow C$ such that
\begin{enumerate}[(i)]
\item $\ker \rho=\ker\theta\rho$, 
\item $\coker \theta\rho$ is zero or has positive $\maxid$-depth,
\item $(\coker\theta\rho)\otimes Q=0$.
\end{enumerate}
\end{lemma}
\begin{proof}
It will be convenient to assume, as we may, that $A=\Hom(R^m,R)$, similarly 
$B=\Hom(R^n,R)$, and $\rho$ is the transpose of a homomorphism 
$\pi:R^n\rightarrow R^m$.
Set $N=\image \pi$. With a principal prime ideal $(p)\subset R$ to be chosen 
later let
$F\subset N$ be as in Lemma \ref{lemma_depth_lemmas_free_loc}.
Define  a homomorphism $\sigma:F\rightarrow R^n$, by specifying its values
on a free generating set, so that $\pi\circ\sigma$ is the inclusion 
$\iota:F\hookrightarrow R^m$.
We will show that with a suitable choice of $p$ we can take
\[
\theta=\sigma^*:\Hom(R^n,R)=B\rightarrow\Hom(F,R)=C.
\]
We summarize the homomorphisms in question in the following commutative diagram
\[
\xymatrix{
F\ar[d]_\sigma \ar@{^{(}->}[dr] \ar[r]^\iota &R^m\\
R^n \ar@{->>}[r]^\pi &N \ar@{^{(}->}[u] \\
}
\qquad
\xymatrix{
C & A\ar[d] \ar[l]_{\iota^*} \\
  B \ar[u]^\theta &\Hom(N,R).\ar@{_{(}->}[l]_-\rho\ar[ul] 
}
\]
First note that, localizing 
\begin{align}\label{eq_localizing}
\image \pi_{(p)}=N_{(p)}=F_{(p)}=\image \iota_{(p)}, \text{ whence }\\
\ker \rho_{(p)}=\ker \pi_{(p)}^*=\ker \iota_{(p)}^*=\ker \theta_{(p)}\rho_{(p)}\subset A_{(p)}. \nonumber
\end{align}
Pulling back by the injective localization map $A\rightarrow A_p$ we obtain
{(i)}.

Next, $N\otimes Q$ is the image of the vector space homomorphism 
$\pi\otimes \id_Q$, while $(\image \rho)\otimes Q$ is the image of its 
transpose $\rho\otimes\id_Q$. It follows that they have the
same dimension. Since 
$N\otimes Q=N_{(p)}\otimes Q=F_{(p)}\otimes Q=F\otimes Q$ by Lemma 
\ref{lemma_depth_lemmas_free_loc}, $(\image \rho)\otimes Q$, $F\otimes Q$, and so
$C\otimes Q$ have the same dimensions. As $\theta\otimes\id_Q$ 
restricts to  a homomorphism 
$(\image \rho)\otimes Q\rightarrow C\otimes Q$
of equidimensional vector spaces, and by {(i)} it is injective, 
{(iii)} follows.

To achieve {(ii)}, we note that if $R$ is a field, then 
it suffices to set $p=0$, for, in that case, $\coker \theta\rho=(\coker \theta\rho)\otimes Q=0$.
So, we will assume that $R$ is not a field.
We pick a nonzerodivisor 
$r\in R$ 
on $(\coker \rho)$, (see Definition \ref{defn_depth}), 
and let $p$ be one of its prime
divisors. Thus, $p\in\maxid$ is a nonzerodivisor on $\coker \rho$. We claim
it is a nonzerodivisor on $\coker \theta\rho$ as well.

Suppose $p$ multiplies the class  in $\coker \theta\rho$ of a $\gamma\in C$
into $0$. This means that $p\gamma=\theta\rho\alpha=\iota^*\alpha$
with some $\alpha\in A=\Hom(R^m,R)$. It follows that the values that 
$\iota^*\alpha$, resp. $\iota^*_{(p)}\alpha_{(p)}$, take are divisible by $p$ in $R$,
resp. $R_{(p)}$. By \eqref{eq_localizing} $\pi_{(p)}^*\alpha_{(p)}$ takes the same values 
as $\iota_{(p)}^*\alpha_{(p)}$. Now for any $s\in R$, $p$ divides $s$ in $R$ precisely
when $p$ divides $(s/1)$ in $R_{(p)}$;
therefore $\pi^*\alpha=\rho\alpha$ is divisible by $p$, say
\begin{equation}
\label{eq_depth_lemmas_div_cond}
\rho\alpha=p\beta, \text{ }\beta\in B.
\end{equation}
Thus, $p\gamma=\theta\rho\alpha=p\,\theta\beta$ and $\gamma=\theta\beta$. On the other hand,
\eqref{eq_depth_lemmas_div_cond} shows that $p$ multiplies the class of $\beta$
in $\coker \rho$ into $0$.
By our choice of $p$, this implies the class of $\beta$ is already $0$,
i.e., $\beta=\image\rho$. Hence $\gamma=\theta\beta\in\image\theta\rho$,
and the class of $\gamma$ in $\coker \theta\rho$ is $0$.
Thus, $p$ is indeed a nonzerodivisor on $\coker \theta\rho$.

But then we are done, since, 
by Nakayama's lemma $\maxid\,\coker\theta\rho\not=\coker\theta\rho$,
unless $\coker\theta\rho=0$.
Therefore, $\coker \theta\rho$ has  positive 
$\maxid$-depth.
\end{proof}

In the next lemma we use the following notation. As before, $\sheaf{O}_0$
is the local ring at $0\in\mathbb{C}^m$, $m\geq1$.
The subring of germs independent of the last coordinate $z_m$ of 
$z\in\mathbb{C}^m$ is denoted $\sheaf{O}'_0$;
the maximal ideals in $\sheaf{O}_0$, $\sheaf{O}'_0$ are $\maxid$, $\maxid'$.
Any $\sheaf{O}_0$-module is automatically an $\sheaf{O}_0'$-module.

\begin{lemma}\label{lemma_depth_lemmas_ind_on_dim}
Suppose $h\in\sheaf{O}_0$ is the germ of a Weierstrass polynomial and
$M$ a finite $\sheaf{O}_0$-module such that $hM=0$. Then
$\depth\maxid M=0$ if and only if $\depth{\maxid'}M=0$.
\end{lemma}
\begin{proof}
We note first that $M$ is a finite $\sheaf{O}_0'$-module. Indeed,
it is an $\sheaf{O}_0/h\sheaf{O}_0$-module, and finite as such.
Since $\sheaf{O}_0/h\sheaf{O}_0$ is also finitely generated as an 
$\sheaf{O}'_0$-module, our claim follows.

If $M=0$, then $\depth\maxid M=\depth{\maxid'}M=\infty$. So, we 
will assume $M\not=0$.
Suppose first 
that $\depth{\maxid'}M=0$. 
We claim that 
there is a nonzero $u\in M$ such that $\maxid' u=0$.
Indeed, this is obvious when $m=1$, since $\sheaf{O}'_0$ is a
field. On the other hand,
when $m\geq 2$, we arrive at this conclusion by applying Proposition 
\ref{prop_depth_lemmas_depth_better}.  

Write $h=z_m^d+\sum_{j=0}^{d-1}a_jz_m^j$, where $a_j\in\maxid'$, $d>0$. As
$u\not=0$ but
\[
	z_m^du=hu-\sum_{j=0}^{d-1}a_juz_m^j=0,
\]
there is a largest $k=0,1,\ldots,k-1$ such that $v=z_m^ku\not=0$.
Then $z_mv=0$, whence $\maxid v=0$, and we conclude by Proposition 
\ref{prop_depth_lemmas_depth_better}, (note, $\sheaf{O}_0$ is not a field).

Conversely, suppose that $\depth{\maxid} M=0$. By Proposition
\ref{prop_depth_lemmas_depth_better}, there is a nonzero submodule $L\subset M$
with $\maxid L=0$. We claim that $\depth{\maxid'} M=0$. 
Indeed, since $\maxid'\subset\maxid$, this is a consequence of
 Proposition \ref{prop_depth_lemmas_depth_better} when $m\geq 2$.
On the other hand, this is obvious if $m=1$, for in this case, $\sheaf{O}_0'$ is a field and
 $M\not=0$.
\end{proof}

\section{The Proof of Theorem \ref{thm_coh_stalk}}
Let us write $(T_m)$ for the statement of Theorem \ref{thm_coh_stalk},
to indicate the number of variables involved. 
We prove it by induction on $m\geq 0$. 
When $m=0$, the depth assumption does not hold, unless $M=0$, 
and so, the claim
is obvious.

When $m=1$, $\sheaf{O}_\zeta$ is a principal ideal domain, and by the corresponding
structure theorem $M$ is a direct sum of a free module and modules of form 
$\sheaf{O}_\zeta/\maxid_\zeta^k$, see \cite[Chapter 12, Theorem 5]{DummitFoote}. If $M$ contained
a submodule of the latter type, then by Proposition 
\ref{prop_depth_lemmas_depth_better} it would have $\maxid$-depth $0$ (take
$L=\maxid_\zeta^{k-1}/\maxid_\zeta^k\subset \sheaf{O}_\zeta/\maxid_\zeta^k$).
Therefore $M$ is free, $p$ has a right inverse $q$ and $\psi=q\phi$
will do.

Now assume $(T_{m-1})$ holds for some $m\geq 2$, and prove $(T_m)$.
We
are free to take $\zeta=0$. Let $Q$ be the field of fractions of $\sheaf{O}_0$.
We first verify $(T_m)$ for torsion modules $M$, i.e., those for
which $M\otimes Q=0$.

Since each generator $v\in M$ is annihilated by some nonzero $h_v\in\sheaf{O}_0$, there is a nonzero $h\in\sheaf{O}_0$ that annihilates all of $M$.
We can assume $h$ is (the germ of) a Weierstrass polynomial of degree
$d\geq 1$ in $z_m$. We write $z=(z',z_m)$ for $z\in\mathbb{C}^m$,
and let $\sheaf{O}_0', \sheaf{O}_0^{\prime F}$ denote the ring/module
of the corresponding germs in $\mathbb{C}^{m-1}$.
As before we embed $\sheaf{O}_0'\subset \sheaf{O}_0$,
$\sheaf{O}_0^{\prime F}\subset\sheaf{O}_0^F$.
This makes any $\sheaf{O}_0$-module an $\sheaf{O}_0'$-module, and any
homomorphism $\phi:N_1\rightarrow N_2$ of $\sheaf{O}_0$-modules descends
to an $\sheaf{O}'_0$-homomorphism 
\begin{equation}\label{eq_proof_descend}
	\phi':N_1/hN_1\rightarrow N_2/h N_2.
\end{equation}
As an example, a version of Weierstrass' division theorem remains true
for holomorphic germs valued in a Banach space $F$ 
(the proof in \cite{GR} applies). Concretely, we can write any $f\in\sheaf{O}^F_0$
uniquely as 
\begin{equation}
\label{eq_proof_weierstrass_div}
f=hf_0+\sum_{j=0}^{d-1}f_j'z_m^j, \qquad f_0\in\sheaf{O}_0^F, f_j'\in\sheaf{O}_0^{\prime F}.
\end{equation}
Clearly, the $\sheaf{O}_0'$-homomorphism
\begin{equation}\label{eq_proof_weierstrass_div_map}
\sheaf{O}_0^F\ni f\mapsto (f_0',\ldots,f'_{d-1})\in(\sheaf{O}_0^{\prime F})^{\oplus d}
\end{equation}
descends to an isomorphism
\begin{equation}
\label{eq_proof_weierstrass_iso}
\sheaf{O}_0^F/h\sheaf{O}_0^F\xrightarrow{\approx}(\sheaf{O}_0^{\prime F})^{\oplus d}
\end{equation}
of $\sheaf{O}_0'$-modules. Composing this with the embedding
\begin{equation}\label{eq_proof_weierstrass_map}
(\sheaf{O}_0^{\prime F})^{\oplus d}\ni (f'_j)\mapsto \sum_j f_j'z_m^j\in\sheaf{O}_0^F,
\end{equation}
we obtain an $\sheaf{O}_0'$-homomorphism
\begin{equation}\label{eq_proof_weierstrass_right_inverse}
\sheaf{O}_0^F/h\sheaf{O}_0^F\rightarrow\sheaf{O}_0^F,
\end{equation}
which is a right inverse of the canonical projection $\sheaf{O}_0^F\rightarrow\sheaf{O}_0^F/h\sheaf{O}_0^F$.

Now $p:\sheaf{O}^n_0\rightarrow M$ and $\phi:\sheaf{O}_0^E\rightarrow M$ of 
the theorem induce $\sheaf{O}_0'$-homomorphisms
\[
	p':\sheaf{O}_0^n/h\sheaf{O}_0^n\rightarrow M,\text{ }
	\phi':\sheaf{O}^E_0/h\sheaf{O}^E_0\rightarrow M,
\]
as in \eqref{eq_proof_descend}, remembering that $hM=0$. 
Clearly, $p'$ is surjective.
Also, by Lemma \ref{lemma_depth_lemmas_ind_on_dim} $\depth{\maxid'} M>0$.
Because of the isomorphism \eqref{eq_proof_weierstrass_iso}, 
$(T_{m-1})$ implies there is an $\sheaf{O}'_0$-homomorphism 
\[
\bar\chi:\sheaf{O}_0^E/h\sheaf{O}_0^E\rightarrow\sheaf{O}_0^n/h\sheaf{O}_0^n
\]
such that $\phi'=p'\bar\chi$.  Since the projection
$\sheaf{O}^n_0\rightarrow\sheaf{O}^n_0/h\sheaf{O}_0^n$ has a right inverse, cf. 
\eqref{eq_proof_weierstrass_right_inverse}, $\bar \chi$ is induced by an $\sheaf{O}_0'$-homomorphism $\chi:\sheaf{O}_0^E\rightarrow\sheaf{O}_0^n$, which
then satisfies $\phi=p\,\chi$. All that remains is to replace $\chi$ by an
$\sheaf{O}_0$-homomorphism $\psi$, which we achieve as follows.

If a holomorphic germ, say $f$, at $0$  valued in a Banach
space $(F,\|\cdot\|_F)$ has a representative on a 
connected neighborhood $V$ of $0$, we write 
\[
	[f]_V=\sup_{v\in V} \|f(v)\|_F\leq\infty,
\]
where $f$ on the right stands for the representative.
Now consider the composition of $\chi|_{\sheaf{O}_0^{\prime E}}$ with
\eqref{eq_proof_weierstrass_div_map}, where $F=\mathbb{C}^n$,
\begin{equation}\label{eq_proof_crazy}
	\sheaf{O}_0^{\prime E}\xrightarrow{\chi}\sheaf{O}^n_0
	\rightarrow\sheaf{O}^{\prime nd}_0.
\end{equation}
By Theorem \ref{thm_plain} this $\sheaf{O}_0'$-homomorphism is induced by a 
$\Hom(E,\mathbb{C}^{nd})$-valued holomorphic function, defined on
some neighborhood $U$ of $0\in\mathbb{C}^{m-1}$. It follows that
if $V'\Subset U$ and $V''\Subset \mathbb{C}$ are connected neighborhoods
of $0\in\mathbb{C}^{m-1}$, resp. $0\in\mathbb{C}$, then there is a constant
$C$ such that for each $e'\in\sheaf{O}^{\prime E}_0$ that has a representative
defined on $V'$
\begin{equation}\label{eq_proof_bdd}
[\chi(e')]_{V'\times V''}\leq C[e']_{V'}.
\end{equation}
Indeed, $\chi$ is obtained by composing \eqref{eq_proof_crazy} with 
\eqref{eq_proof_weierstrass_map} 
(again, $F=\mathbb{C}^n$), and this latter is trivial to estimate.

Now define $\psi:\sheaf{O}^E_0\rightarrow\sheaf{O}_0^n$ by
$\psi(e)=\sum_{j=0}^\infty\chi(e'_j)z_m^j$,
if
$e=\sum_{k=0}^\infty e_j'z_m^j\in\sheaf{O}_0^E$.
Cauchy estimates for $e_j'$ and \eqref{eq_proof_bdd}
together imply that the series above indeed represents a germ $\psi(e)\in\sheaf{O}^n_0$.
It is straightforward that $\phi$ is an $\sheaf{O}_0$-homomorphism. Because 
of this, $p\psi=\phi$ holds on $h\sheaf{O}_0^E$, both sides being zero.
It also holds on polynomials $e=\sum_{j=0}^k e_j'z_m^j$, as
\[
	(p\psi)(e)=p\sum_{j}\chi(e_j')z_m^j=\sum_j\phi(e_j')z_m^j=\phi(e).
\]
The division formula \eqref{eq_proof_weierstrass_div}, this time with $F=E$,
now implies $p\psi=\phi$ on all $\sheaf{O}_0^E$.

Having taken care of torsion modules, consider a general module $M$ as
in the theorem. 
Since $\sheaf{O}_0$ is Noetherian, $\ker p$ is finitely generated;
let 
\[
	\rho:\sheaf{O}_0^r\rightarrow \sheaf{O}_0^n
\]
have image $\ker p$. So, $M\approx \coker \rho$. Construct a free 
$\sheaf{O}_0$-module $C$ together with a homomorphism
 $\theta:\sheaf{O}_0^n\rightarrow C$ as in Lemma 
\ref{lemma_depth_lemmas_reduce_to_torsion}.
Let $\pi:C\rightarrow\coker \theta\rho$ be the canonical projection.
Here is a diagram to keep track of all the homomorphisms in question:
\begin{equation*}
\xymatrix{
\sheaf{O}_0^r\ar[d]_\rho\\
\sheaf{O}^n_0  \ar[d]_p\ar[r]^\theta &C \ar[r]^-\pi &\coker\theta\rho\\
M & &\sheaf{O}_0^E. \ar[ll]_\phi \ar[u]_{\tilde{\phi}} \ar@{-->}[lu]|{\tilde\psi}\\
}
\end{equation*}
We are yet to introduce $\tilde{\phi}$, $\tilde{\psi}$. For $e\in\sheaf{O}_0^E$
choose $v\in \sheaf{O}_0^n$ so that $p(v)=\phi(e)$.
Then $\pi\theta(v)$ is independent of which $v$ we choose, 
since any two choices differ by an element of $\ker p=\image\rho$, 
which $\pi\theta$ then maps to $0$. We let $\tilde{\phi}(e)=\pi\theta(v)$.
We want to lift $\tilde{\phi}$ to $C$; this certainly can be done if 
$\coker\theta\rho=0$. Otherwise Lemma \ref{lemma_depth_lemmas_reduce_to_torsion}
guarantees that $\depth\maxid\coker\theta\rho>0$ and 
$(\coker\theta\rho)\otimes Q=0$. Hence we can apply the first part of this proof to obtain a homomorphism 
$\tilde{\psi}:\sheaf{O}^E_0\rightarrow C$
such that $\pi\tilde{\psi}=\tilde{\phi}$.

Finally, we lift $\tilde{\psi}$ to $\sheaf{O}^n_0$ as follows.
For $e\in\sheaf{O}^E_0$ choose $v\in\sheaf{O}^n_0$ and $w\in\sheaf{O}^r_0$
so that
$\phi(e)=p(v)$ and $\tilde{\psi}(e)=\theta(v)+\theta\rho(w)$.
Again $v+\rho(w)\in\sheaf{O}^n_\zeta$ is independent of the choices.
(It suffices to verify this for $e=0$. Then $v\in\ker p=\image\rho$; let $v\in\rho(u)$, $u\in\sheaf{O}^r_0$.
Hence $0=\theta(v)+\theta\rho(w)=\theta\rho(u+w)$. By Lemma \ref{lemma_depth_lemmas_reduce_to_torsion}
this implies $0=\rho(u+w)=v+\rho w$ as claimed.)

Therefore we can define a homomorphism $\psi:\sheaf{O}^E_0\rightarrow\sheaf{O}^n_0$
by letting $\psi(e)=v+\rho w$. Since $p\psi(e)=p(v)=\phi(e)$,
$\psi$ is the homomorphism we were looking for.

We conclude this section by showing that the depth condition in Theorem 
\ref{thm_coh_stalk} is also necessary.

\begin{theorem}\label{thm_necessary}
Let $m,n\geq 1$, $\zeta\in\mathbb{C}^m$, $M\not=0$ a finite
$\sheaf{O}_\zeta$-module, and $p:\sheaf{O}^n_\zeta\rightarrow M$ an
epimorphism. If $\depth{\maxid_\zeta} M=0$, then for any 
infinite dimensional Banach space $E$ there is a homomorphism $\phi:\sheaf{O}^E_\zeta\rightarrow M$ that does not factor through $p$.
\end{theorem}
\begin{proof}
By Proposition \ref{prop_depth_lemmas_depth_better}, $M$ has a submodule 
$N\not=0$ for which $\maxid_\zeta N=0$. This implies that $N$ is a finite dimensional vector space over $\mathbb{C}\subset\sheaf{O}_\zeta$ ($\mathbb{C}$
embedded as constant germs).
Consider the homomorphism of $\mathbb{C}$-vector spaces 
\[
	\epsilon:\sheaf{O}_\zeta^E\rightarrow E,
	\text{ } \epsilon(e)=e(0),
\]
and a $\mathbb{C}$-linear map $l:E\rightarrow N$. The composition 
$l\epsilon:\sheaf{O}^E_\zeta\rightarrow N$ is a homomorphism of $\sheaf{O}_\zeta$-modules.
If it can be written as $p\psi$, where  $\psi:\sheaf{O}^E_\zeta\rightarrow\sheaf{O}_\zeta^n$
is a homomorphism of$\sheaf{O}_\zeta$-modules, then $\psi$ is plain by Theorem
\ref{thm_local_modules} and, by looking at how $\psi$ acts on constant germs,
we see that $l$ must be continuous. Therefore, by taking $l:E\rightarrow N$
to be linear and discontinuous, we obtain a nonfactorizable homomorphism $\psi=l\epsilon$.
\end{proof}

\section{ A Local Theorem }
\begin{theorem}\label{thm_local_thm}
Let $\sheaf{S}$ be a coherent sheaf over an open $\Omega\subset\mathbb{C}^m$,
such that the $\maxid_\zeta$-depth of each nonzero stalk is positive.
Suppose $p:\sheaf{O}^n\rightarrow \sheaf{S}$ is an epimorphism
and $E$ is a Banach space.
Then 
any $\sheaf{O}$-homomorphism $\sheaf{O}^E\rightarrow\sheaf{S}$
factors 
through $p$ in some neighborhood of $\zeta$. 
\end{theorem}
This theorem is a special case of a stronger, global, statement,
whose proof, however, involves a cohomology vanishing theorem for
infinitely generated sheaves.
Since the local version yields a few immediate applications, we will post-pone
the proof of the stronger theorem until Section 8.

The local statement is a simple consequence of Theorem \ref{thm_local_modules}
once the following two auxiliary statements are proved:

\begin{lemma}\label{lemma_epimorphisms_reduce_to_torsion}
Let $\rho:\sheaf{A}\rightarrow \sheaf{B}$ be a homomorphism of finite free
$\sheaf{O}$-modules over $\Omega\subset\mathbb{C}^m$ and $\zeta\in \Omega$. 
Denote by $Q$ the field of fractions of $\sheaf{O}_\zeta$.
Then, there are a finite free $\sheaf{O}$-module $\sheaf{C}$, a neighborhood $U$ of $\zeta$, and a homomorphism $\theta:\sheaf{B}|_U\rightarrow \sheaf{C}|_U$, such that
\begin{enumerate}[(i)]
\item $\ker \rho|_U=\ker (\theta\rho)|_U$,
\item $(\coker \theta\rho)_\zeta\otimes Q=0$.
\end{enumerate}
\end{lemma}
\begin{proof}
We just repeat the proof of Lemma \ref{lemma_depth_lemmas_reduce_to_torsion}.
Let $\sheaf{N}=\simage \rho^*$, where 
\[
\rho^*:
\shom(\sheaf{B},\sheaf{O})\rightarrow\shom(\sheaf{A},\sheaf{O})
\]
 is the map dual
to $\rho$.
Since $\sheaf{N}_\zeta\otimes Q$ is a finite dimensional 
vector space
over $Q$, there is a finite free $\sheaf{O}_\zeta$-module 
$\sheaf{F}_\zeta\hookrightarrow \sheaf{N}_\zeta$ such that $\sheaf{F}_\zeta\otimes Q=N_\zeta\otimes Q$. Define 
$\sigma_\zeta: \sheaf{F}_\zeta \rightarrow 
		\Hom(\sheaf{B}_\zeta,\sheaf{O}_\zeta)$
by specifying its values on a free generator set so that
$\rho_\zeta^*\sigma_\zeta$ is the inclusion $\iota_\zeta:\sheaf{F}_\zeta\hookrightarrow\Hom(\sheaf{A}_\zeta,\sheaf{O}_\zeta)$.
Take 
\[
	\theta_\zeta=\sigma^*_\zeta:
 		\sheaf{B}_\zeta\approx 
		\Hom(\Hom(\sheaf{B}_\zeta,\sheaf{O}_\zeta),\sheaf{O}_\zeta)
 			\rightarrow
		\Hom(\sheaf{F}_\zeta,\sheaf{O}_\zeta)=\sheaf{C}_\zeta.
\]
Now, since $\sheaf{C}_\zeta$ is a free $\sheaf{O}_\zeta$-module, we can consider
it as a stalk of a sheaf $\sheaf{C}$ of finite free $\sheaf{O}$-modules.
The stalk homomorphism $\theta_\zeta:\sheaf{B}_\zeta\rightarrow\sheaf{C}_\zeta$
is induced by a homomorphism-valued holomorphic map $\bar\theta$ on some neighborhood $U$ of $\zeta$. 
Hence $\theta_\zeta$ extends to an  $\sheaf{O}$-homomorphism
$\theta:\sheaf{B}|_U\rightarrow\sheaf{C}|_U$. 

Since $\otimes$ is a right-exact covariant functor, 
\[
\image (\rho^*_\zeta\otimes\id_{Q})
	=\sheaf{N}_\zeta\otimes Q=\sheaf{F}_\zeta\otimes Q=\image (\iota_\zeta\otimes\id_{Q}),
\]
 and so,
\[
\ker (\rho_\zeta\otimes\id_{Q})=\ker (\iota^*_\zeta\otimes\id_{Q})=\ker (\theta\rho)_\zeta\otimes\id_{Q}.
\]
Noting that $\sheaf{A}_\zeta\rightarrow\sheaf{A}_\zeta\otimes Q$
is injective, $\ker \rho_\zeta=\ker (\theta\rho)_\zeta$. On the other hand
both $\sker \rho$ and $\sker\theta\rho$ are coherent sheaves, hence,
 after shrinking $U$ we may assume {(i)} holds.
We also note that
\[
\begin{split}
  \dim \sheaf{C}_\zeta\otimes Q=
  \dim \sheaf{F}_\zeta\otimes Q=
  \dim \sheaf{N}_\zeta\otimes Q=\\ 
  =\dim\, \image (\rho_\zeta^*\otimes\id_{Q})=
  \dim\, &\image (\rho_\zeta\otimes\id_{Q}).
\end{split}
\]
Consequently, the injectivity of the finite vector space homomorphism 
\[
\theta_\zeta\otimes\id_{Q}|_{
	\image(\rho_\zeta\otimes\id_{Q})}:
\image(\rho_\zeta\otimes\id_{Q})\hookrightarrow 
\sheaf{C}_\zeta\otimes Q
\]
implies its surjectivity, and {(ii)} follows immediately.
\end{proof}

\begin{lemma}\label{lemma_epimorphisms_mono}
Let $\sheaf{S}$ be a coherent sheaf over $\Omega$. If $\zeta\in\Omega$, then
there is a neighborhood $U\subset\Omega$ of $\zeta$ so that $\sheaf{S}(U)\rightarrow\sheaf{S}_\zeta$ is a monomorphism.
\end{lemma}
\begin{proof}
We follow the outline given by the proof of 
 Theorem \ref{thm_local_modules}. 
The lemma holds trivially for the sheaves over $\Omega=\mathbb{C}^0=0$.
For $\Omega$ lying in higher dimensions we proceed by induction.
Initially,
we verify 
the inductive step for torsion
modules; then, the general case is proved by reduction to a torsion case. 

We are free to take $\zeta=0$.
As before $Q$ denotes the field of quotients of $\sheaf{O}_0$.
Suppose $\Omega\subset\mathbb{C}^m$, $m\geq 1$, and $\sheaf{S}_0\otimes
Q=0$. 
Then $\sheaf{S}_0$ is annihilated by a nonzero $h\in\sheaf{O}_0$.
After choosing a suitable neighborhood 
$U=U'\times U''\subset\mathbb{C}^{m-1}\times\mathbb{C}$ of $0$
to be reduced further later, 
we can take $h$ to be a Weierstrass
polynomial of degree $d\geq 1$ in $z_m$ and $h\sheaf{S}|_U=0$. 
We write $z=(z',z_m)$ and $\sheaf{O}'$
for the sheaf of germs in $\mathbb{C}^{m-1}$. 

Let $|A|=\{z\in U'\times U'': h(z)=0\}$,  $|B|=U'\times \{0\}$, and 
$\sheaf{O}_A=(\sheaf{O}/h\sheaf{O})|_{|A|}$.
Define complex spaces $A=(|A|,\sheaf{O}_A)$ and $B=(|B|,\sheaf{O}')$. 
Since $h\sheaf{S}|_{U}=0$, $\supp \sheaf{S}|_U\subset|A|$ and
$\sheaf{S}|_{|A|}$ has a structure of an $\sheaf{O}_A$-module.  
If 
$\sheaf{O}^r|_V\xrightarrow{\rho}\sheaf{O}^n|_V\xrightarrow{p}\sheaf{S}|_V\rightarrow 0$ 
is an exact sequence
for some $V\subset U$,
then 
the induced sequence 
$\sheaf{O}_A^r|_{|A|\cap V}\xrightarrow{\rho}\sheaf{O}_A^n|_{|A|\cap V}
\xrightarrow{p}\sheaf{S}|_{|A|\cap V}\rightarrow 0$
is also exact. So, $\sheaf{S}|_{|A|}$ is also $\sheaf{O}_A$-coherent.  
The projection $U'\times U'' \rightarrow U'$ induces a holomorphic Weierstrass map 
$\pi:A\rightarrow B$, see \cite[Section 2.3.4]{GR}. Since $\pi$ is a finite map,
the direct image sheaf $\pi_*(\sheaf{S}|_{|A|})$ is a coherent sheaf
 over $U'\subset\mathbb{C}^{m-1}$.

Inductively we can assume that $U'$ is such that 
$\pi_*(\sheaf{S}|_{|A|})(U')\rightarrow\pi_*(\sheaf{S}|_{|A|})_0$ 
is a monomorphism.
On the other hand 
\[
\pi_*(\sheaf{S}|_{|A|})(U')=\sheaf{S}|_{|A|}(\pi^{-1}U')
=\sheaf{S}(U), \text{ and } \pi_*(\sheaf{S}|_{|A|})_0=\prod_{\xi\in\pi^{-1}(0)}\sheaf{S}_\xi=\sheaf{S}_0,
\]
see \cite[Section 2.3.3]{GR}. So, $\sheaf{S}(U)\rightarrow\sheaf{S}_0$ is a monomorphism.

Now consider a general coherent sheaf $\sheaf{S}$. On some neighborhood
$U$ of $\zeta$ there exists an exact sequence 
$
\sheaf{O}^r\xrightarrow{\rho}\sheaf{O}^n\xrightarrow{p} \sheaf{S}|_U\rightarrow 0.
$
This neighborhood $U$ can be taken so that there exists
 a homomorphism $\theta:\sheaf{O}^n|_U\rightarrow\sheaf{O}^s|_U$ as in Lemma
\ref{lemma_epimorphisms_reduce_to_torsion}. Since $(\scoker \theta\rho)_\zeta$
is a torsion $\sheaf{O}_\zeta$-module, we can apply the first part of the proof
to assume that
\begin{equation}\label{eq_epimorphisms_coker_mono}
(\scoker \theta\rho)(U)\rightarrow(\scoker\theta\rho)_\zeta 
\text{ is a monomorphism.}
\end{equation}
Furthermore, we can take $U$ to be a pseudoconvex domain.

Suppose $s\in\sheaf{S}(U)$ and $s_\zeta=0$. Let $v\in\sheaf{O}^n(U)$ be such
that $p(v)=s$. Then $v_\zeta\in\image \rho_\zeta$ and
the residue of $\theta_\zeta v_\zeta$ 
 vanishes in 
$(\scoker \theta\rho)_\zeta$.  
 By \eqref{eq_epimorphisms_coker_mono}, $\theta v$ represents a zero
section in $\scoker \theta\rho|_U$, i.e., there is $w\in\sheaf{O}^r(U)$ with 
$\theta v=\theta\rho w$. Since $\sker \rho|_U=\sker \theta\rho|_U$,
$v=\rho w$ and, thus, $s=0$ 
\end{proof}
\begin{proof}[Proof of Theorem \ref{thm_local_thm}]
Let $\phi:\sheaf{O}^E\rightarrow\sheaf{S}$ be an $\sheaf{O}$-homomorphism.
If 
$\zeta\in\Omega$, then 
according to Theorem \ref{thm_local_modules}  there is
a plain homomorphism 
$\psi^\zeta:\sheaf{O}^E_\zeta\rightarrow \sheaf{O}^n_\zeta$ 
so that 
\begin{equation} \label{eq_local_phi_p_psi}
	\phi|_{\sheaf{S}_\zeta}=p \psi^\zeta. 
\end{equation}
Since $\psi^\zeta$ is induced by a homomorphism-valued holomorphic map,
$\psi^\zeta$ 
extends to a plain homomorphism 
$\psi_U:\sheaf{O}^E|_U\rightarrow\sheaf{O}^n|_U$
for a neighborhood  $U\subset\Omega$  of $\zeta$.
By Lemma \ref{lemma_epimorphisms_mono}, we can assume that 
$\sheaf{S}(U)\rightarrow\sheaf{S}_\zeta$ is a monomorphism. 
In conjunction with \eqref{eq_local_phi_p_psi}, this implies that 
$\phi(v)-p\psi_U(v)=0$ for $v\in \sheaf{O}^E(U)$, in particular,  
for $v$ a constant section. Then, an application of Lemma 
\ref{lemma_im_is_small} shows that
 $\image (\phi_{\zeta'}-(p\psi_U)_{\zeta'})\subset \bigcap_{j=0}^\infty \maxid_{\zeta'}\sheaf{S}_{\zeta'}=0$ for $\zeta'\in U$, i.e.,
that $\phi$ factors through $p$ on $U$.
\end{proof}

\section{ Applications}
Our first application is Corollary \ref{cor_non_free}. It depends on the 
following
\begin{proposition}\label{prop_applications_seq}
Suppose $(R,\maxid)$ is a local ring with the residue field
$k=R/\maxid$, and $M$ is a free $R$-module. If
$c_\nu\in k$ for $\nu\in \mathbb{N}$   and $e_\nu\in M$ are such that
their classes $\bar e_\nu$ in $M/\maxid M$ are linearly independent over $k$
then there is an $R$-homomorphism $\phi:M\rightarrow R$ such that
for every $\nu$ the class of $\phi(e_\nu)$ in $k$ is $c_\nu$.
\end{proposition}
\begin{proof}
Let $\bar\phi:M/\maxid M\rightarrow k$ be a $k$-linear map such that
$\bar\phi(\bar e_\nu)=c_\nu$. Composing $\bar\phi$ with the projection $M\rightarrow M/\maxid M$ we obtain an $R$-homomorphism $\psi:M\rightarrow k$
such that $\psi(e_\nu)=c_\nu$. If $M$ is free then 
$\psi$ can be lifted to a $\phi:M\rightarrow R$ as required.
\end{proof}

\begin{proof}[Proof of Corollary \ref{cor_non_free}]
Let $e_\nu\in\sheaf{O}_\zeta^E$ be germs such that $e_\nu(\zeta)\in E$
are $\mathbb{C}$-linearly independent unit vectors.
Any $\sheaf{O}_\zeta$-homomorphism 
$\phi:\sheaf{O}^E_\zeta\rightarrow\sheaf{O}_\zeta$
is plain by Theorem \ref{thm_local_modules}, whence $\phi(e_\nu)(\zeta)\in\mathbb{C}$
is a bounded sequence. If $\sheaf{O}^E_\zeta$ were a submodule of a 
free module $M$ then by Proposition \ref{prop_applications_seq}
there would exist a homomorphism $\sheaf{O}_\zeta^E \rightarrow\sheaf{O}_\zeta$
such that $\phi(e_\nu)(\zeta)=\nu$, a contradiction.
\end{proof}

For further applications we have to review some concepts introduced in 
\cite{LP}.
As there, in this review we place ourselves in an open subset $\Omega$ of a
Banach space $X$; but our applications will only concern finite
dimensional $X$.

In the Introduction we have already defined plain sheaves and homomorphisms.
For sheaves $\sheaf{A},\sheaf{B}$ of $\sheaf{O}$-modules (always over $\Omega$)
we write 
$\shomo(\sheaf{A},\sheaf{B})$ for the sheaf of $\sheaf{O}$-homomorphisms 
between them; if $\sheaf{A}$ and $\sheaf{B}$ are plain sheaves we write
$\shomp(\sheaf{A},\sheaf{B})\subset\shomo(\sheaf{A},\sheaf{B})$ for the sheaf of plain homomorphisms.

\begin{definition}\label{defn_applications_anal_str}
An analytic structure on a sheaf $\sheaf{S}$
is the choice, for each plain sheaf $\sheaf{E}$, of a submodule 
$\shom(\sheaf{E},\sheaf{S})\subset\shomo(\sheaf{E},\sheaf{S})$ subject to 
\begin{itemize}
\item If $\sheaf{E},\sheaf{F}$ are plain sheaves, $x\in\Omega$, and 
	$\varphi\in \shomp(\sheaf{E},\sheaf{F})_x$, then\\
	$\varphi^*\shom(\sheaf{F},\sheaf{S})_x
	 \subset\shom(\sheaf{E},\sheaf{S})_x$; and
\item $\shom(\sheaf{O},\sheaf{S})=\shomo(\sheaf{O},\sheaf{S})$.
\end{itemize}
\end{definition}

If $\sheaf{S}$ is endowed with an analytic structure, one also says that 
$\sheaf{S}$ is an analytic sheaf. This terminology is different from
the traditional one, where ``analytic sheaves'' and ``sheaves of $\sheaf{O}$-modules''  mean one and the same thing.

If $U\subset\Omega$ is open, an $\sheaf{O}$-homomorphism 
$\psi:\sheaf{S}|_U\rightarrow\sheaf{S}'|_U$ of analytic sheaves
is called analytic if 
$\psi_*\shom(\sheaf{E}|_U,\sheaf{S}|_U)\subset
\shom(\sheaf{E}|_U,\sheaf{S}'|_U)$
for every plain sheaf
$\sheaf{E}$.

Any plain sheaf $\sheaf{F}$ has a canonical analytic structure 
given by
$\shom(\sheaf{E},\sheaf{F})=\shomp(\sheaf{E},\sheaf{F})$. 
Further, on any $\sheaf{O}$-module $\sheaf{S}$ one can define a
``maximal'' analytic structure by $\shom(\sheaf{E},\sheaf{S})=\shomo(\sheaf{E},\sheaf{S})$;
and also a ``minimal'' analytic structure, denoted by $\shommin(\sheaf{E},\sheaf{S})$,
consisting of germs $\alpha$ that can be written  as a composition $\beta\gamma$ 
of 
\[
\gamma\in\shomp(\sheaf{E},\sheaf{O}^n)\text{ and }\beta\in\shomo(\sheaf{O}^n,\sheaf{S}),
\]
where $n<\infty$. Definition \ref{defn_applications_anal_str} implies
that 
\[
	\shommin(\sheaf{E},\sheaf{S})\subset
	\shom(\sheaf{E},\sheaf{S})\subset
	\shomo(\sheaf{E},\sheaf{S}).
\]
In view of Theorems \ref{thm_plain} and \ref{thm_local_thm} we obtain 
the following uniqueness results
\begin{theorem}
For every plain sheaf $\sheaf{O}^F$ over  an open  $\Omega\subset\mathbb{C}^m$, $0<m<\infty$, the canonical and the maximal analytic structures coincide.
\end{theorem}
\begin{proof}
Let $E$ be a Banach space, $U\subset\Omega$ an open set, and 
$\phi:\sheaf{O}^E|_U\rightarrow\sheaf{O}^F|_U$ an $\sheaf{O}$-homomorphism. 
By Theorem \ref{thm_plain},
$\phi$ is a plain homomorphism, and hence, an analytic homomorphism for the
canonical analytic structure.
Thus, $\shom(\sheaf{O}^E,\sheaf{O}^F)=\shomo(\sheaf{O}^E,\sheaf{O}^F)$.
\end{proof}
\begin{theorem}\label{thm_coh_unique_anal_str}
Let $\sheaf{S}$ be a coherent sheaf such that $\depth{\maxid_\zeta} \sheaf{S}_\zeta>0$ for
$\zeta\in\supp \sheaf{S}$. Then the minimal and the maximal analytic structures
coincide, i.e., $\sheaf{S}$ has unique analytic structure.
\end{theorem}
\begin{proof}
Denote by $\Omega\subset\mathbb{C}^m$ the base of the sheaf $\sheaf{S}$.
Let $E$ be a Banach space,  $U\subset\Omega$ an open set, and 
$\phi:\sheaf{O}^E|_U\rightarrow\sheaf{O}^F|_U$ an $\sheaf{O}$-homomorphism.
If $m=0$, the depth condition guarantees that $\sheaf{S}=0$ and the conclusion
of the theorem
follows. So, we may assume that $m\geq 1$.

Since $\sheaf{S}$ is a coherent sheaf, given $\zeta\in U$ there is 
an epimorphism $p:\sheaf{O}^n|_V\rightarrow\sheaf{S}|_V$, with $n<\infty$ and
$V\subset U$, a suitable neighborhood of $\zeta$.
By Theorem \ref{thm_local_thm}, we can assume that  $\phi|_V$
factors through $p|_V$, i.e., 
there is an $\sheaf{O}$-homomorphism 
$\psi:\sheaf{O}^E|_V\rightarrow\sheaf{O}^n|_V$ with $\phi|_V=p|_V\psi$.
Then, by Theorem \ref{thm_plain}, $\psi$ is a plain homomorphism, and so, 
$\phi_\zeta\in\shommin(\sheaf{O}^E,\sheaf{O}^F)_\zeta$. Since 
$\phi$ and $\zeta$ were arbitrary, it follows that $\shommin(\sheaf{O}^E,\sheaf{O}^F)=\shomo(\sheaf{O}^E,\sheaf{O}^F)$. 
\end{proof}

\section{ Epimorphisms on Coherent Sheaves }

\begin{theorem}\label{thm_coh_global}
Let $\sheaf{S}$ be a coherent sheaf over an open pseudoconvex 
$\Omega\subset\mathbb{C}^m$, and $E$ be a Banach space.
Suppose 
 $\depth{\maxid_\zeta}\sheaf{S}_\zeta>0$ 
for
$\zeta\in\supp \sheaf{S}$.  
If $p:\sheaf{O}^n\rightarrow\sheaf{S}$
is an epimorphism, then any $\sheaf{O}$-homomorphism $\sheaf{O}^E\rightarrow\sheaf{S}$ 
factors through it.
\end{theorem}
\begin{proof}[Proof of Theorem \ref{thm_coh_global}]
This proof is based on the following key fact \cite[Theorem 4.3]{Lempert} due to
Lempert:
a coherent sheaf over $\Omega\subset\mathbb{C}^m$ endowed with its minimal 
analytic structure is cohesive. While \cite{Lempert} makes references to some
of the results of the present paper,  
the proof of \cite[Theorem 4.3]{Lempert} is 
independent of Theorem \ref{thm_coh_global}.

Let $\phi:\sheaf{O}^E\rightarrow\sheaf{S}$ be an $\sheaf{O}$-homomorphism.
In view of Theorem \ref{thm_local_thm}, there is an open pseudoconvex cover $\mathfrak{V}$ of $\Omega$
such that on each $V\in\mathfrak{V}$ there is a homomorphism 
$\psi_V:\sheaf{O}^E|_V\rightarrow \sheaf{O}^n|_V$ with $\phi|_V=p\psi|_V$.
If we let $\sheaf{K}=\sker p$, a coherent sheaf,
then $\psi_{VW}=\psi_V-\psi_W$ maps $\sheaf{O}^E_{V\cap W}$
into $\sheaf{K}$, for $V,W\in\mathfrak{V}$.
Thus, the $\sheaf{O}$-homomorphisms $\psi_{VW}$ form a $\sheaf{K}$-valued
$1$-cocycle.

We can assume that $m\geq 1$, for otherwise, the depth condition implies that $\sheaf{S}$ is a zero-sheaf and there is
nothing to prove.
The module 
$\sheaf{K}\subset\sheaf{O}^n$ is torsion-free, i.e., 
$r_\zeta k_\zeta\not=0$
for $\zeta\in\Omega$, $r_\zeta\in\sheaf{O}_\zeta$, and $k_\zeta\in\sheaf{K}_\zeta$,
unless $r_\zeta=0$ or $k_\zeta=0$. Therefore, 
in view of 
Lemma \ref{prop_depth_lemmas_depth_better},
$\depth{\maxid_\zeta} \sheaf{K}_\zeta>0$ for all $\zeta\in\supp \sheaf{K}$.
We endow $\sheaf{K}$ with the minimal analytic structure, and note that, 
 by Theorem 
\ref{thm_coh_unique_anal_str},
$\psi_{VW}$ are analytic with respect to this structure.
On the other hand, $\sheaf{K}$ is coherent, and hence, by \cite[Theorem 4.3]{Lempert}, is cohesive. 
Now
$H^1(\Omega,\shom(\sheaf{O}^E,\sheaf{K}))=0$,
which is a special case of \cite[Theorem 9.1]{LP}.
Consequently,
$\psi_{VW}=\theta_V-\theta_W$ with some (analytic)
homomorphisms $\theta_V:\sheaf{O}^E|_V\rightarrow\sheaf{K}|_V$;
defining $\psi:\sheaf{O}^E\rightarrow\sheaf{O}^n$ by
\[
	\psi|_V=\psi_V-\theta_V,
\]
the resulting homomorphism satisfies $\phi=p\psi$.
\end{proof}


\bibliographystyle{amsalpha}
\bibliography{higas}

\providecommand{\bysame}{\leavevmode\hbox to3em{\hrulefill}\thinspace}
\providecommand{\MR}{\relax\ifhmode\unskip\space\fi MR }
\providecommand{\MRhref}[2]{%
  \href{http://www.ams.org/mathscinet-getitem?mr=#1}{#2}
}
\providecommand{\href}[2]{#2}
\begin{thebibliography}{Mat80}

\bibitem[DF99]{DummitFoote}
David~S. Dummit and Richard~M. Foote, \emph{Abstract algebra}, New York: Wiley,
  1999.

\bibitem[Eis99]{eisen}
David Eisenbud, \emph{Commutative algebra: with a view toward algebraic
  geometry}, Springer, 1999.

\bibitem[GR84]{GR}
Hans Grauert and Reinhold Remmert, \emph{Coherent analytic sheaves}, Springer,
  1984.

\bibitem[Lem]{Lempert}
L\'aszl\'o Lempert, \emph{Coherent sheaves and cohesive sheaves}, manuscript.

\bibitem[LP07]{LP}
L\'aszl\'o Lempert and Imre Patyi, \emph{Analytic sheaves in {B}anach spaces},
  Ann. Sci. \'Ecole Norm. Sup. \textbf{40} (2007), no.~3, 453--486.

\bibitem[Mat80]{matsumura}
Hideyuki Matsumura, \emph{Commutative algebra}, Benjamin/Cummings Pub. Co.,
  1980.

\bibitem[Muj86]{mujica}
Jorge Mujica, \emph{Complex analysis in {B}anach spaces: holomorphic functions
  and domains of holomorphy in finite and infinite dimensions}, Elsevier
  Science Pub. Co., 1986.

\bibitem[Ser55]{serre}
Jean-Pierre Serre, \emph{Faisceaux alg\'ebriques coh\'erents}, Ann. Math.
  \textbf{61} (1955), 197--278.

\end{thebibliography}
\end{document}